\documentclass[english,11pt]{smfart}

\usepackage[utf8]{inputenc}
\usepackage[T1]{fontenc}
\usepackage[swedish, english]{babel}

\usepackage{tikz-cd}
\usepackage{etex}
\usepackage{booktabs}

\newcommand{\BGW}{\operatorname{BGW}}
\newcommand{\tree}{\mathrm{Tree}}
\newcommand{\map}{\mathrm{Map}}
\usepackage{amsbsy}
\usepackage{amsmath,amsfonts,amssymb,amsthm,mathrsfs,mathtools}

\usepackage{bm}

\usepackage[a4paper,vmargin={3cm,3cm},hmargin={3.5cm,3.5cm}]{geometry}
\usepackage[font=sf, labelfont={sf,bf}, margin=1cm]{caption}
\usepackage{graphicx}
\usepackage{epsfig}
\usepackage{latexsym}
\usepackage{xcolor}
\usepackage{ae,aecompl}
\usepackage{soul,framed}
\usepackage{comment}

\usepackage{xcolor}
\usepackage[pdfpagemode=UseNone,bookmarksopen=false,colorlinks=true,urlcolor=blue,citecolor=blue,citebordercolor=blue,linkcolor=blue]{hyperref}
\usepackage{smfhyperref}
\usepackage[capitalize]{cleveref}

\usepackage{pstricks}
\usepackage{enumerate}
\usepackage{tikz,animate,media9}						
\usepackage{todonotes}
\usepackage{pifont}
\usepackage{bm,marvosym}
\usepackage{algorithm}
\usepackage{algorithmic}
\usepackage{subcaption}

\usepackage{bbm}

\usepackage{tcolorbox}

\definecolor{aleacolor}{rgb}{0.16,0.59,0.78}

\hypersetup{
	breaklinks,
	colorlinks=true,
	linkcolor=aleacolor,
	urlcolor=aleacolor,
	citecolor=aleacolor}

\newcount\colveccount
\newcommand*\colvec[1]{
	\global\colveccount#1
	\begin{pmatrix}
		\colvecnext
	}
	\def\colvecnext#1{
		#1
		\global\advance\colveccount-1
		\ifnum\colveccount>0
		\\
		\expandafter\colvecnext
		\else
	\end{pmatrix}
	\fi
}

\newcommand{\south}{\ast_0}
\newcommand{\north}{\ast_1}
\newcommand{\out}{\operatorname{out}}

\newcommand{\ndN}{\mathbb{N}}

\newcommand{\SP}{\mathcal{SP}}

\renewcommand{\Pr}[1]{\mathbb{P}(#1)}

\newcommand{\one}{{\mathbbm{1}}}

\newcommand{\dis}{\operatorname{dis}} 

\newcommand{\convdis}{\,{\buildrel d \over \longrightarrow}\,}

\newcommand{\He}{\mathrm{H}}
\newcommand{\he}{\mathrm{h}}

\newcommand{\Di}{\mathrm{D}}

\newcommand{\cE}{\mathcal{E}}

\newcommand{\cG}{\mathcal{G}}

\newcommand{\cN}{\mathcal{N}}

\newcommand{\cP}{\mathcal{P}}

\newcommand{\cR}{\mathcal{R}}
\newcommand{\cS}{\mathcal{S}}
\newcommand{\cT}{\mathcal{T}}
\newcommand{\cU}{\mathcal{U}}

\newcommand{\Seq}{\textsc{SEQ}}

\newcommand{\mN}{\mathsf{N}}

\newcommand{\mP}{\mathsf{P}}

\newcommand{\mR}{\mathsf{R}}
\newcommand{\mS}{\mathsf{S}}
\newcommand{\mT}{\mathsf{T}}

\newcommand{\mSP}{\mathsf{SP}}

\newcommand{\oo}{\infty}

\newtheorem{theorem}{Theorem}[section]

\newtheorem{corollary}[theorem]{Corollary}
\newtheorem{proposition}[theorem]{Proposition}
\newtheorem{lemma}[theorem]{Lemma}

\newtheorem{definition}[theorem]{Definition}
\newtheorem{remark}[theorem]{Remark}
\newtheorem{example}[theorem]{Example}

\numberwithin{equation}{section}

\title[The scaling limit of random 2-connected series-parallel maps]
{The scaling limit of random 2-connected series-parallel maps}
\author[Amankwah]{Daniel Amankwah}
\email{daa33@hi.is}
\author[Björnberg]{Jakob Björnberg}
\email{jakob.bjornberg@gu.se}
\author[Stef\'ansson]{Sigurdur \"Orn Stef\'ansson}
\email{sigurdur@hi.is}
\author[Stufler]{Benedikt Stufler}
\email{benedikt.stufler@tuwien.ac.at}
\author[Turunen]{Joonas Turunen}
\email{joonas.turunen@alumni.helsinki.fi}

\begin{document}

\vspace {-0.5cm}

\maketitle

\begin{abstract}
A finite graph embedded in the plane is called a
series-parallel map if it can be obtained from a
finite tree by repeatedly subdividing and doubling
edges. We study the scaling limit of weighted random
two-connected series-parallel maps with $n$ edges and
show that under some integrability conditions on these
weights, the maps with distances rescaled by a factor
$n^{-1/2}$ converge to a constant multiple of Aldous'
continuum random tree (CRT) in the Gromov--Hausdorff
sense. The proof relies on a bijection between a set
of trees with $n$ leaves and a set of series-parallel
maps with $n$ edges, which enables one to compare
geodesics in the maps and in the corresponding trees
via a Markov chain argument introduced by Curien, Haas
and Kortchemski (2015).  
\end{abstract}
	
	\maketitle

	\begin{figure}
		\begin{subfigure}{0.25\textwidth}
			\centering
			\includegraphics[width=\linewidth]{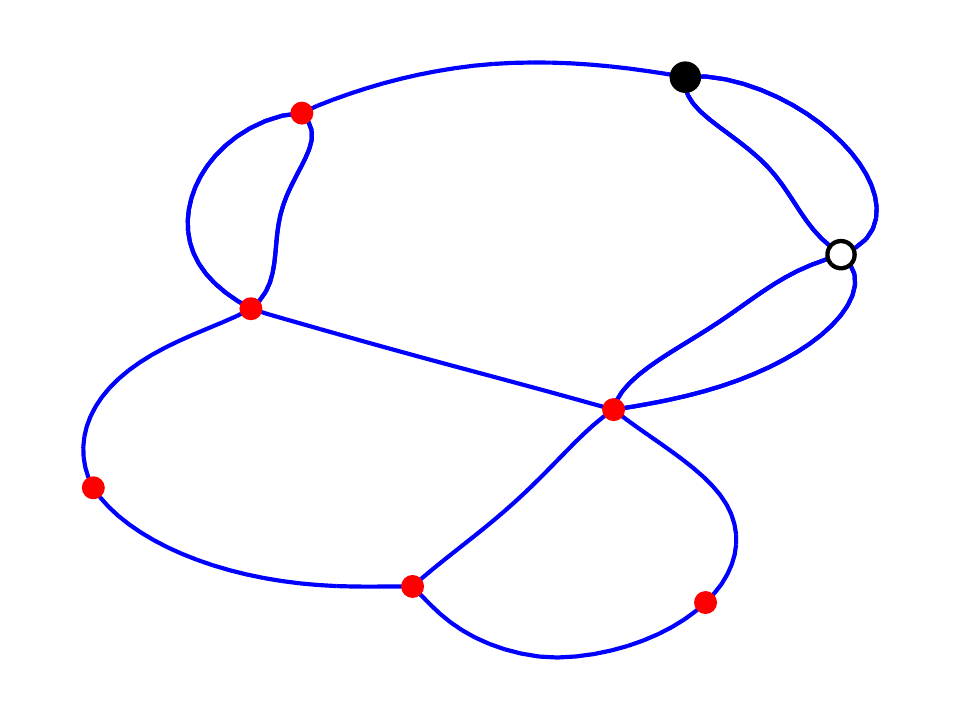}
			\caption{13 edges.}
		\end{subfigure}%
		\hfill
		\begin{subfigure}{0.36\textwidth}
			\centering
			\includegraphics[width=\linewidth]{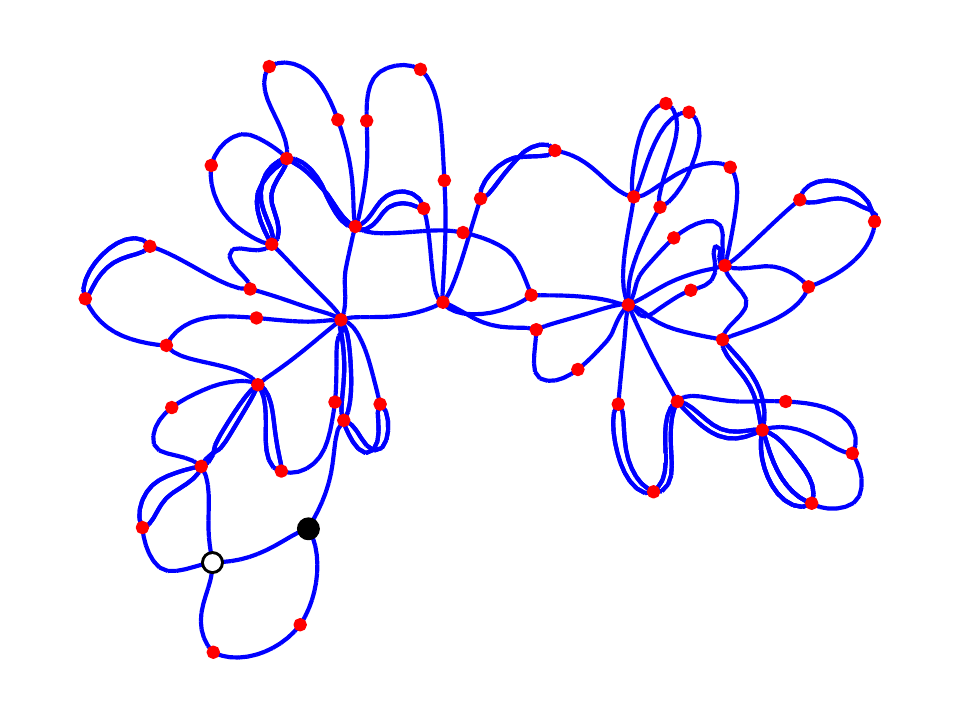}
			\caption{102 edges.}
		\end{subfigure}%
		\hfill
		\begin{subfigure}{0.39\textwidth}
			\centering
			\includegraphics[width=\linewidth]{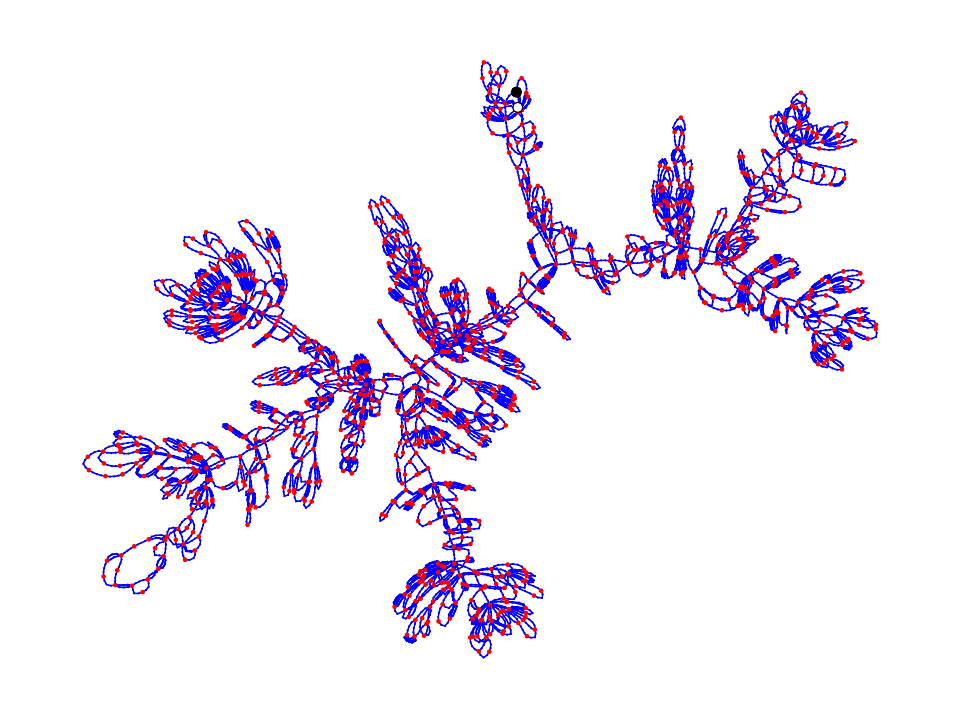}
			\caption{2352 edges.}
		\end{subfigure}
		\caption{A simulation of uniform 2-connected series-parallel maps with different numbers of edges. The endpoints of the root edge are represented by $\bullet$ and $\circ$.}\label{f:sim}
	\end{figure}

\section{Introduction}\label{sec:intro}

A finite connected graph is called a \emph{series-parallel graph}
if it can be obtained from a finite tree by repeatedly subdividing or
doubling edges.   It may equivalently be characterized as a graph that does not contain the complete graph $K_{4}$  as a minor. In this work we consider
series-parallel \emph{maps}, which are
series-parallel graphs embedded in the plane,
viewed up to continuous deformations.

Specifically, we consider series-parallel
maps which are \emph{two-connected}, and we
are interested in their \emph{scaling limits}. 
Without the  condition of two-connectivity,
uniformly sampled series-parallel maps have
Aldous' continuum
random tree (CRT) as their scaling limit, as
easily follows from a general theorem of
Stufler \cite[Thm.~6.62]{stufler:2020}.
More preciesely, uniform samples of such  maps with 
$n$ edges, and distances  rescaled by $n^{-1/2}$,
converge weakly in the Gromov--Hausdorff sense
to a multiple of the CRT. 
The essence of the proof lies in decomposing the map into
two-connected components, or \emph{blocks}, and using
a coupling with simply
generated trees where each block of the
map corresponds to a vertex in the tree. As the tree scales at the
order $n^{1/2}$, and
since blocks are typically small, after rescaling distances by
$n^{-1/2}$  on average each block
contributes a constant length to the geodesic.  Thus, the global shape of
the rescaled map is approximated by the tree, stretched by a constant factor. 

In this work, however, we would like to understand the structure of
the blocks themselves.
Thus, we focus our attention on \emph{two-connected} series-parallel
maps.  These do not enjoy the same type of coupling to simply generated
trees and prove more difficult to handle. Our main result is that,
nevertheless, uniformly sampled two-connected
series-parallel maps with $n$ edges converge, as $n\to\oo$ with
distances rescaled by $n^{-1/2}$,  
to a multiple of the CRT.  Figure \ref{f:sim} shows simulations of
these uniformly sampled maps and suggests that they become tree-like
with an increasing number of edges.

Analogous results have been established before for several tree-like
random maps, which can often be characterized through excluded
minors. 
Most fundamental is the original result of Aldous, showing that
uniform trees with $n$ edges converge to the CRT when rescaled by
$n^{-1/2}$ \cite{Aldous1,Aldous2,Aldous3}.   
Trees may be characterized by not containing $K_3$ as a minor.
Other notable cases
are various classes of {outerplanar graphs}. 
A graph is called \emph{outerplanar} if it may be
embedded in the plane so that every vertex lies on the boundary of
the external face;   equivalently, they are graphs  not containing $K_4$ or the
complete bipartite graph $K_{3,2}$ as minors. In \cite{caraceni},
Caraceni established that uniformly sampled outerplanar maps (properly
rescaled) converge to the CRT. Stufler  \cite{stufler:2020}
generalized this result to a family of face-weighted outerplanar maps,
using the aforementioned coupling with simply generated trees. 
Note that outerplanar maps are series-parallel maps,
since series-parallel maps have 
$K_4$ as an excluded minor, and outerplanar maps 
further exclude  $K_{3,2}$ as a minor.

If one additionally imposes the condition of two-connectivity
on outerplanar maps, 
one obtains \emph{dissections of a polygon}.  
Random, face weighted dissections of polygons were
studied 
by Curien, Haas and Kortchemski \cite{randissection} who showed, under
a suitable condition on the weights, that the scaling limit is a
constant multiple of the CRT. Their proof involves relating the
dissections bijectively to trees and using a clever Markov chain
argument to compare lengths of geodesics in the dissections and their
corresponding trees.

One of the main aims of the above studies is to understand how
universally the CRT appears as a scaling limit for different models of
random graphs or maps. 
Note that planar maps themselves,
given by excluding $K_5$ and $K_{3,3}$ as minors, 
behave vastly differently.  Uniform planar maps with $n$ edges, when
rescaled by $n^{-1/4}$, admit the so-called \emph{Brownian sphere} as
a scaling limit \cite{bettinelli:2014} which is different from the
CRT.    See Table \ref{tab:my_table} for a summary of these results.  
\begin{table}[htbp]
\centering
{\footnotesize
\begin{tabular}{cccc}
  \toprule
  Map & Excluded minors & Diameter &  Scaling limit \\
  \midrule
  Plane trees & $K_3$ & $n^{1/2}$   & CRT \\
  Outerplanar maps  & $K_{4}$, $K_{2,3}$ & $n^{1/2}$  &  CRT \\
  Series-parallel maps & $K_4$ &$n^{1/2}$  & CRT \\
  Planar maps & $K_5, K_{3,3}$ & $n^{1/4}$ &   Brownian sphere \\
  \bottomrule
\end{tabular}
\caption{Scaling limits of uniform maps with $n$ edges characterized
  by exclusion of minors.} \label{tab:my_table}
}
\end{table}

The proof of our main scaling limit result 
follows the strategy developed by Curien, Haas and
Kortchemski  for random dissections 
\cite{randissection}.   We use a bijection
with certain plane trees (see Section \ref{ss:bijection} for details), 
which enables us
to describe the distribution of the random maps with $n$ edges in
terms of a Bienaymé--Galton--Watson (BGW) tree conditioned on having $n$
leaves. We then adapt the Markov chain argument from
\cite{randissection} for comparing geodesics in the map 
and in the corresponding tree. 

Moreover, the relation with BGW-trees allows us
to consider a weighted version of the random maps in a natural way by
assigning a general offspring distribution $\mu$ to the BGW-trees. 
We will assume that this offspring distribution is critical and that it
has finite exponential moments  (is light-tailed), which includes the uniform
distribution as a special case. The precise definition of the weighted
series-parallel maps is given in Definition \ref{def:randommap} in
Section \ref{ss:weights}. 

The version of CRT used in this work is
the one constructed from a normalised Brownian excursion $\mathbf{e}$
(see \cite{LeGall1}) denoted by $\mathcal{T}_\mathbf{e}$. 
Let $\mSP^\mu_n$ denote the random, weighted, rooted, $2$-connected
series-parallel map with $n$ non-root edges corresponding to a
critical offspring distribution $\mu$ of a BGW-process possessing
finite exponential moments (see Definition \ref{def:randommap}).  
The following  is
our main result: 

\medskip

\begin{theorem}
	\label{te:main1}
	There exists a constant $c_{\mu}>0$ such that
	\[
	(\mSP^\mu_n, c_{\mu} n^{-1/2} d_{\mSP_n}) \convdis (\cT_{\mathrm{e}}, d_{\cT_{\mathrm{e}}})
	\]
	in the Gromov--Hausdorff sense.
\end{theorem}

\medskip

We also obtain optimal tail-bounds for the diameter $\Di(\mSP^\mu_n)$:

\medskip

\begin{theorem}
\label{te:main2}
There exist constants $C,c>0$ such that for all $n$ and all $x>0$
\[
  \Pr{\Di(\mSP^\mu_n) > x} \le C \exp(-c x^2 /n).
\]
\end{theorem}

\medskip

\begin {remark}
We will consider other closely related classes of series-parallel maps referred to as series-parallel networks, their
definitions are given in Subsection \ref{ss:spdef}. The two previous theorems also apply to random weighted series-parallel networks with the same scaling constant $c_\mu$.
\end {remark}

\section{Decomposing and sampling series-parallel maps} \label{s:setup}

\subsection{Rooted, two-connected 
series-parallel  maps} 
\label{ss:spdef} 

A planar map is a drawing of a finite
connected graph in the plane so that edges do not cross except
possibly at their endpoints, viewed up to orientation preserving
homeomorphisms of the plane. A series-parallel (SP) map is a planar
map which does not contain the complete graph $K_4$ as a minor.  The
objects of interest in this paper are two-connected SP-maps. Another
way to describe them is to start from a double edge and repeatedly
subdivide or double edges. 

We will assign a
direction to one of the edges of the map 
and refer to it as the \emph{root} of the map.
The root will be denoted by $e=(\ast_0,\ast_1)$. For concreteness, we
draw the maps  so that the root edge lies
on the boundary of the unbounded face and is directed in a clockwise
manner along that face. See Fig.~\ref{f:spexample} (left) for an
example. 
\begin{figure}[h]
	\centering
	\includegraphics[width=0.5\linewidth]{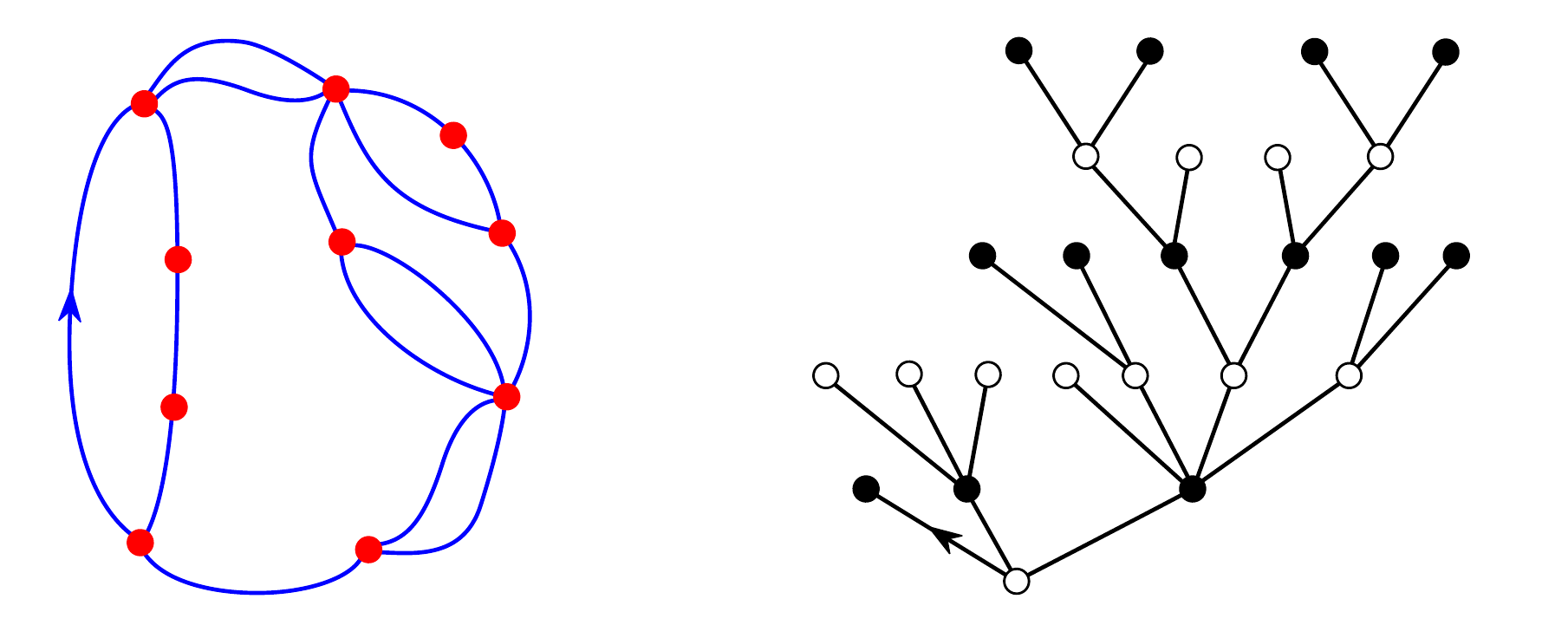}
	\caption{Left: An example of a rooted, two-connected SP-map. The directed root edge is indicated by an arrow. Right: Its corresponding labeled plane tree where $\circ$ denotes the label $P$ and $\bullet$ denotes the label $S$.}
	\label{f:spexample}
\end{figure}
Denote the set of rooted two-connected SP-maps with $n \geq 1$
non-root edges by $\SP_{n}$ and the set of all such finite maps by
$\SP=\cup_{n\geq 1} \SP_n$. We deliberately leave out the map with the
single root edge later on for notational convenience . 

In order to be
able to describe a decomposition of the maps into series and parallel
components, we require two additional definitions. Let $M \in \SP_{n}$
and let $M'$ be the map obtained by removing the root edge from
$M$, keeping the endpoints $\ast_0,\ast_1$ as marked vertices of $M'$. Define  
for all $n \geq 2$
\begin{align*}
  \cS_n &= \{M' ~|~ M \in \SP_{n}, M' \text{ is not two-connected}\}, \\
  \cP_n &= \{M' ~|~ M \in \SP_{n}, M' \text{ is two-connected} \},
\end{align*}
Also set 
\begin{equation}\label{eq def N}
\cS = \cup_{n\geq 2} \cS_n, \qquad 
\cP = \cup_{n\geq 2}\cP_n,\qquad
 \text{and} \qquad
\cN= \cS\cup \cP \cup\{e\}.
\end{equation}
The operation of removing or adding back the root edge defines a bijection 
\begin{align} \label{eq:removerootmap}
	\partial_\map: \SP \to \cN,
\end{align}
see Fig.~\ref{f:spremove}.
We refer to the  maps in $\cN$  as \emph{series-parallel (SP)
networks}. 

\begin{figure}[htb]
	\centering
	\includegraphics[width=0.85\linewidth]{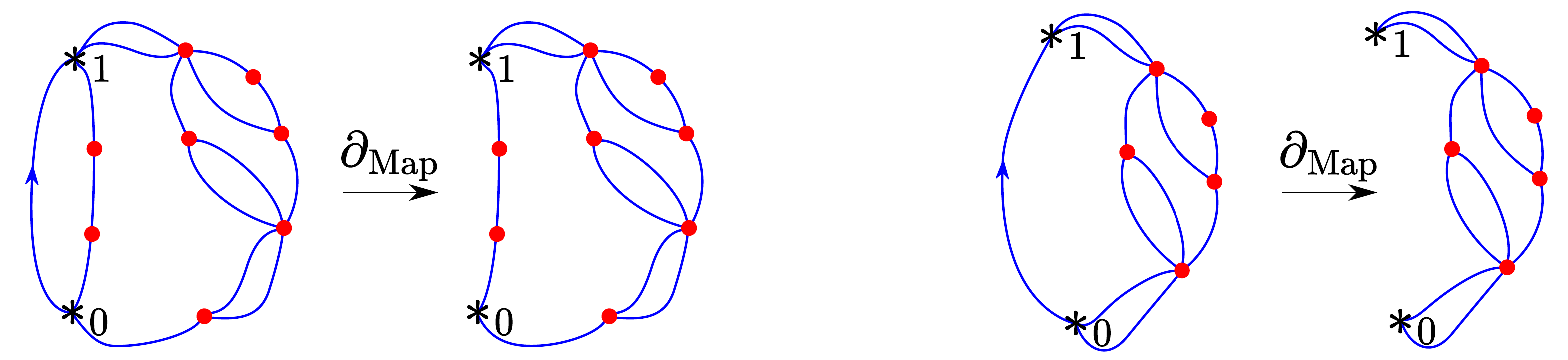}
	\caption{The root edge is removed from a rooted two-connected SP-map. The resulting map is either two-connected (left) or it is not (right).}
	\label{f:spremove}
\end{figure}

\subsection{Plane trees and labels}
\label{sec:planetrees}

 Recall the standard definition of the Ulam--Harris tree: set
$\mathcal{U}=\bigcup_{n=0}^{\infty} \mathbb{N}^{n},$
where
$\mathbb{N}=\{1,2,3,\dotsc\}$ and 
 $\mathbb{N}^{0}=\{\varnothing\}$ by convention.  
If $u=u^{1},\dots,u^{m}$ and $v=v^{1},\dots,v^{n}$ are
elements of $\mathcal{U}$, we write the concatenation of $u$ and $v$
as $uv=u^{1},\dots,u^{m},v^{1},\dots,v^{n}$.\\
\begin{definition}
A \emph{plane tree} $T$ is a finite or infinite subset of
$\mathcal{U}$ such that
\begin{enumerate}
\item $\varnothing \in T$,
\item if $ v \in T$ and $v=uj$ for some $j \in \mathbb{N}$ then $u \in T$.
\item for every $u \in T $, there exists an integer $\out_T(u) \geq 0$ (the outdegree or number of children of $u$) such that for every $j \in \mathbb{N}$, $uj \in T$ if and only if $1 \leq j \leq \out_T(u)$.  
\end{enumerate}
\end{definition}

\medskip

The element $\varnothing\in\cU$ is called the 
\emph{root vertex} and the pair
$\{\varnothing,1\}$ is called the \emph{root edge}. One may
view plane trees as planar maps with a directed root edge
which do not contain a cycle. A vertex of outdegree $0$ is called a
\textit{leaf} and an edge incident to a leaf is called a \textit{leaf
  edge}. The trivial tree consisting only of the root vertex
$\varnothing$, which is then a leaf, will be denoted by $\bullet$. 

We denote by $\cT$ the set of \emph{finite}
plane trees which have no vertices of
outdegree 1, and by $\cT_n\subseteq\cT$ the set of such trees
with $n\geq1$ leaves. 
We will apply labels $S$ and $P$ to the trees in $\cT$ 
(related to the series and
parallel decompositions of maps).  We use the notation
$\bar{S} = P$ and $\bar{P} = S$ when we need to swap labels. If
$(T,\ell)\in \cT \times \{S,P\}$ and $T\neq \bullet$ we say that $T$
is labeled by $\ell$ and we assign the label $\ell$ to vertices in
every even generation and $\bar \ell$ to vertices in every odd
generation. In the case when $T=\bullet$ we take $(\bullet,P)$ and
$(\bullet,S)$ to be equivalent and simply denote them by $\bullet$ and
say that $\bullet$ is neither labeled by $P$ nor $S$. Define $\cT\cP =
(\cT\setminus\{\bullet\})\times \{P\}$ and $\cT\cS =
(\cT\setminus\{\bullet\})\times \{S\}$ as the sets of nontrivial trees
labeled by $P$ and $S$, respectively, and let 
\begin{equation}\label{eq def TN}
\cT\cN := \cT\cP\cup
\cT\cS  \cup \{\bullet\}.
\end{equation} 
The corresponding sets of trees with $n$
leaves are denoted using the lower index $n$.  

Next define $\cT^\ast\subseteq\cT$ as the set of trees in
$\cT$ such that the leftmost
child of the root is a leaf, and let $\cT^\ast_n$ denote the set of
such trees with $n+1$ 
leaves (i.e.\ $n$ leaves not counting the leftmost child of the root). A tree $T\in \cT^\ast$ is commonly called a \emph{planted tree}. {We take trees in $\cT^\ast$ to be labelled so that vertices in even generations have label $P$ and vertices in odd generations have label $S$.
There is a useful bijection
\begin{align} \label{eq:removeroottree}
	\partial_{\tree}: \cT^\ast \to \cT{\cN} 
\end{align}
which is described as follows. Let $T\in \cT_{n}^\ast$. If the
outdegree of the root in $T$ is strictly greater that 2, remove the leftmost child
of the root (along with
the edge adjacent to it).  Note that the resulting tree is still labeled
by $P$, so this yields an element in $\cT\cP_n$. If the outdegree of
the root in $T$ equals 2, 
then remove
the root and its leftmost child along with their adjacent edges. The
resulting tree is labeled by $S$ and is thus an element in $\cT\cS_n$,
see Fig.~\ref{f:bijection_trees}. In the case when $T$ has exactly 1
leaf, let $\partial_\tree T = \bullet$. 
\begin{figure}[h]
	\centering
	\includegraphics[width=1\linewidth]{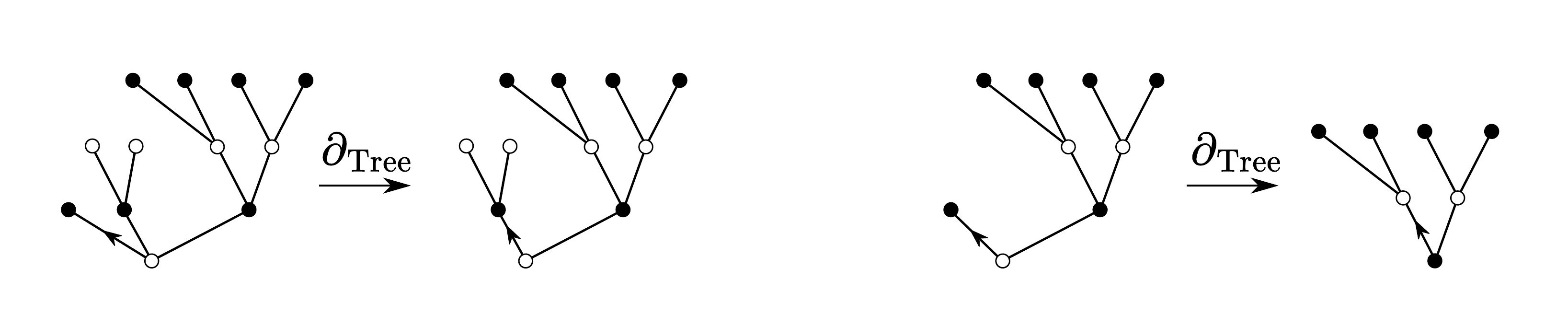}
	\caption{The removal of the first child of the root. Here, the vertices $\circ$ correspond to label $P$ and the vertices $\bullet$ to label $S$.}
	\label{f:bijection_trees}
\end{figure}

\subsection{Bijection between SP-maps 
and labeled trees}\label{ss:bijection}

Both SP-networks  $\mathcal N$
\eqref{eq def N}
and labelled trees   $\mathcal {TN}$
\eqref{eq def TN}
 may be described in the language of
combinatorial species \cite{Joyal: 1981}.
This provides a natural framework for defining 
bijections $\varphi:\cT\cN\to\cN$
and $\varphi^\ast: \cT^\ast\to\SP$, as we now describe.

First, an element of $\mathcal N$
  is either the single edge $e = (*_0,*_1)$,
or a parallel network or a series network: 
\begin{align} \label{eq:networkmain}
	\cN &=  e + \cP + \cS.
\end{align}
The elements of $\cP$ are called parallel networks and are constructed
from a parallel composition of at least two non-parallel networks: 
\begin{align} \label{eq:networkP}
	\cP &= \Seq_{\ge 2} (\cN - \cP)
\end{align}
where the combinatorial species $\Seq_{\geq 2}$ are the linear orders
of length greater than or equal to two.  The elements of $\cS$ are
called series networks and are constructed from a series composition
of at least two non-series networks: 
\begin{align} \label{eq:networkS}
	\cS &= \Seq_{\ge 2} ( \cN - \cS).
\end{align}
The single edge network is treated seperately in the above definitions
since it is in some sense neither series nor parallel.

Similarly, a tree from $\cT{\cN}$ is either the unlabeled single leaf
$\bullet$ or a tree labeled by $P$ or a tree labeled by $S$: 
\begin{align} \label{eq:treemain}
	\cT{\cN} &= \bullet + \cT\cP+ \cT\cS.
\end{align}
A tree labeled by $P$ is constructed by a composition of at least 2
trees from $\cT_{\cN}$ which are not labeled by $P$: 
\begin{align} \label{eq:treeP}
	\cT\cP &= \Seq_{\ge 2} (	\cT{\cN}  - 	\cT\cP).
\end{align}
A tree labeled by $S$ is constructed by a composition of at least 2
trees from $\cT_{\cN}$ which are not labeled by $S$: 
\begin{align} \label{eq:treeS}
	\cT\cS &= \Seq_{\ge 2} (\cT{\cN}  - \cT\cS).
\end{align}

The combinatorial identities \eqref{eq:networkmain},
\eqref{eq:networkP}, \eqref{eq:networkS} and \eqref{eq:treemain},
\eqref{eq:treeP}, \eqref{eq:treeS} naturally define a
bijection $$\varphi:\cT\cN\to\cN.$$ Moreover, networks with $n$ edges
are in bijection with the set of labeled trees with $n$ leaves and
networks in $\cP$ and $\cS$ are in bijection with labeled trees in
$\cT\cP$ and $\cT\cS$ respectively. 
Combining $\varphi$ with the bijections
\eqref{eq:removerootmap} and \eqref{eq:removeroottree} allows us to
define a bijection 
$\varphi^\ast: \cT^\ast\to\SP$
as in Fig.~\ref{f:diagram}. 
\begin{figure}[!h]
\[
\begin{tikzcd}
\cT^\ast \arrow{r}{\varphi^\ast} \arrow[swap]{d}{\partial_\tree} & 
\SP \arrow{d}{\partial_\map} \\
\cT\cN \arrow{r}{\varphi} & \cN
\end{tikzcd}
\]
\caption{Definition of $\varphi^\ast$ from $\varphi$.} \label{f:diagram}
\end{figure}

It will be useful to have the following recursive description of 
$\varphi^\ast$.  Given a labelled tree $(T,\ell)\in\cT\times\{S,P\}$,
first assign a single edge to each of the leaves.
Then, start forming  networks by combining the edges corresponding to
sibling leaves in series or in 
parallel, as determined by the label of their common parent.
Proceeding recursively towards the root of the tree,
each vertex corresponds to a series or a parallel network
which is constructed from the subtree of its descendants.
Networks corresponding to vertices with a common parent are
connected in series if that parent has label $S$, respectively 
 in parallel if it has label $P$.
To be definite, parallel connections  are made
in the same order left-to-right as the  order of the
corresponding siblings,  while series connections are made so that the
left-to-right order of the siblings gives the bottom-to-top
order of connections. 
Restricting this construction to $\cT^\ast$ gives the bijection
$\varphi^\ast$.

\subsection{Blobs.}

We now define what we  call \emph{blobs}.
 Informally, blobs are
 parts of $M$ 
 which are separated by bottlenecks, in the sense
that a geodesic between two vertices within the same blob stays
entirely within that blob.  See Fig.~\ref{f:spexample2} for an
illustration of the description that follows.
\begin{figure}[h]
  \centering
  \includegraphics[width=0.9\linewidth]{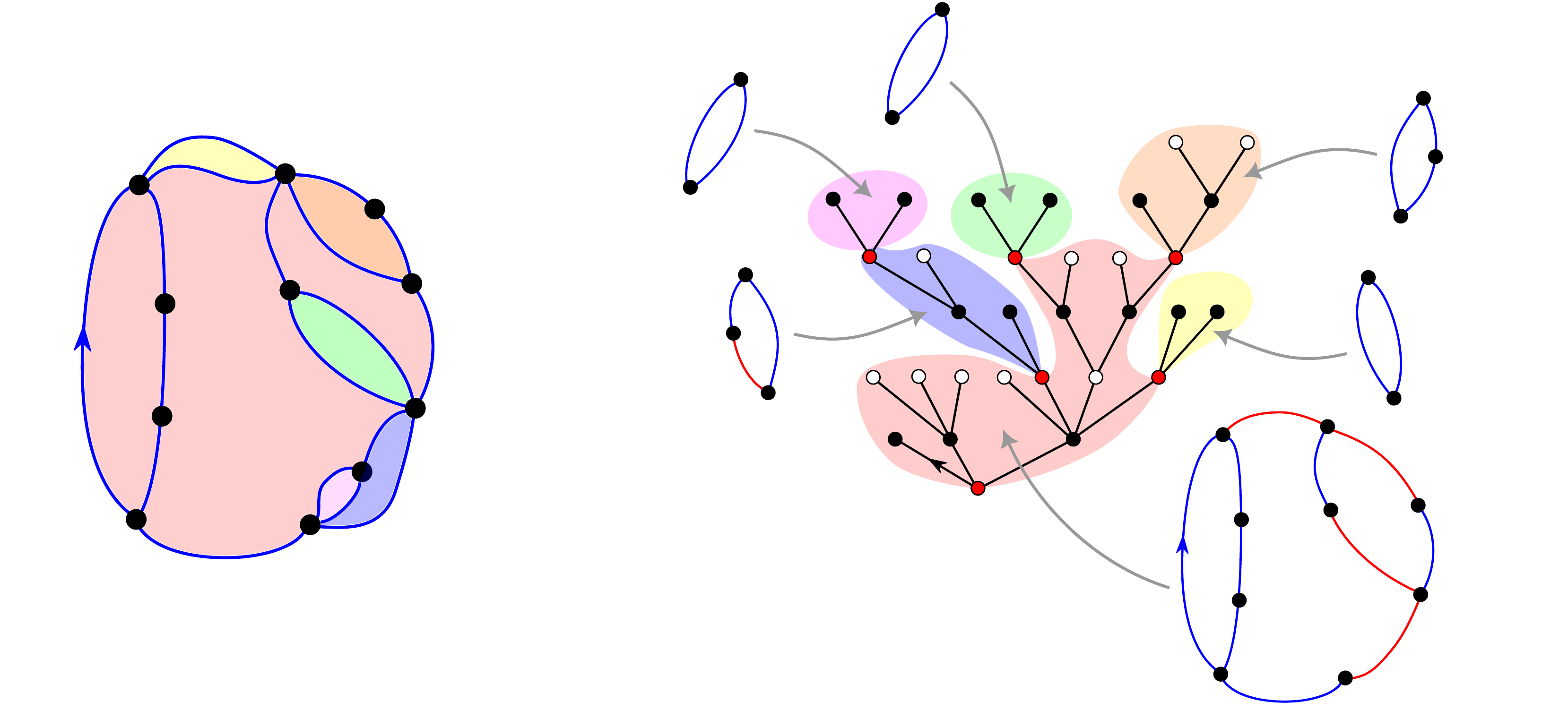}
  \caption{Left: An example of a rooted, two-connected SP-map divided
    into blobs. Right: Its corresponding planted plane tree, labeled
    by $P$, divided into segments. Here, $\circ$ and
    {\color{red}$\bullet$} denote the label $P$ and $\bullet$ denotes
    the label $S$. Each segment corresponds to a network and the red
    edges depicted in these networks are the places where networks
    corresponding to fringe subtrees are inserted.} 
	\label{f:spexample2}
\end{figure}

Suppose that we are given a rooted $2$-connected SP-map $M\in\cS\cP$
and the corresponding planted tree
$T^\ast=(\varphi^\ast)^{-1}(M)\in\cT^\ast$ (which we recall is
labeled by $P$).  
We colour a vertex of $T^\ast$ red if it is marked $P$ and has at
least one child that is a leaf. In particular, the root of $T^\ast$ is
coloured red. This way, the tree $T^\ast$ is decomposed into
(maximal) subtrees whose roots are red, and
 where any red non-root
vertex is a leaf.  We call these subtrees \textit{segments}. 

We define the \emph{blobs}  of $M$ to be the networks obtained by
applying the recursive construction defining $\varphi^\ast$ (described
at the end of Section \ref{ss:bijection}) to the segments.
We moreover declare the blobs to have an oriented red
placeholder edge for each red non-root vertex in the segment, which
marks the location where we would insert the network constructed from
the corresponding fringe subtree in $T^\ast$.

The orientation of each red edge is naturally inherited from the orientation of
the edge in the tree from the root to the red vertex, together with the orientation of the root edge of the SP-map, where the tree is explored in its usual lexicographic order. 
 We assign placeholder south and north
 poles to each of the blobs as the start- and endpoints of the
oriented edge that corresponds to the red edge assigned to the red leaf of the corresponding parent segment.
This way, the map $M$ gets decomposed into
blobs, where the placeholder poles of each blob are connected by
an edge in $M$.

 We perform a similar decomposition of networks $N$ into blobs,
 where all definitions are analogous except that we demand the root
 vertex of the corresponding tree $T$ to be coloured red.

\subsection{Bienaym\'e--Galton--Watson trees 
and simply generated trees}  

The main focus of this work is uniformly distributed
two-connected SP-maps. However, the bijective correspondence with trees
which is explained in the previous section allows us easily to
consider a more general measure on the maps by assigning weights to
the trees in a natural way and pushing these forward to the maps using
the bijection.

Let $\mu = (\mu(k))_{k\geq 0}$ be a sequence of probabilities with
$\mu(1) = 0$.  Let $\BGW^\mu$ denote the distribution of a
Bienaymé--Galton--Watson process with offspring distribution $\mu$. If
$\mu$ has mean less than or equal to $1$ then the tree $\tau$ is
almost surely finite and, for any finite tree $T$,
\begin{equation}
  \BGW^{\mu}(\tau=T) = \prod_{u \in T}\mu(\out_\tau(u)).
\end{equation}
Let $\BGW^{\mu}_n$ be the corresponding measure  conditioned on
the number of leaves being $n$.  The random trees distributed by
$\BGW^\mu$ and $\BGW^\mu_n$ will be denoted by $\mT^\mu$ and
$\mT^\mu_n$ respectively.

In order to relate the BGW-trees to the random maps, it will be useful
to consider another (almost) equivalent model of random trees called
\emph{simply generated trees with leaves as atoms}. Let $\eta =
(\eta(i))_{i\geq 0}$ be a sequence of non-negative numbers called
\emph{weights} and assume that $\eta(1) = 0$. Let $T\in \cT_n$ and
define the weight of $T$ by 
\begin{align*}
	\eta(T) = \prod_{u\in T} \eta(\out_T(u)).
\end{align*}
A \emph{simply generated tree} on $n$ leaves is a random plane tree $\tau$ with values in $\cT_n$ with a probability distribution
\begin{align} \label{simplygen}
	\mathbb{P}(\tau = T) = \frac{\eta(T)}{\sum_{T'\in \cT_n} \eta(T')}, \qquad T \in \cT_n.
\end{align}
The literature contains many results on simply generated trees with a
fixed number of vertices, and for these we refer to \cite{janson} for
details. The current case, where instead the number of leaves is fixed, is
considered in e.g. \cite{stufler:2020}.  Below we outline the
main results which we require.

For $a,b > 0$, let $\eta^{a,b}$ be another weight
sequence defined by  
\begin{align}\label{eq:tilt1}
	\eta^{a,b}(k)= \begin{cases}
		a \eta(0) & k = 0 \\
		b^{k-1} \eta(k) & k\geq 2.
	\end{cases}
\end{align}
These weights are equivalent to $\eta$ in the sense that they define
the same measure as in \eqref{simplygen}. Define the generating series \begin {align} \label{partitionfunction}
\qquad g^\eta(z) = \sum_{n=0}^\infty \eta(n) z^n
\end{align}
and denote its radius of convergence by $\rho_\eta$. We will assume
that the weights $\eta$ are chosen such that $\rho_\eta > 0$. Then
for any $b>0$ such that $\frac{g^\eta(b)-\eta(0)}{\eta(0)  b} < 1$
one may choose 
\begin{align*}
a = \frac{1}{\eta(0)}-\frac{g^\eta(b)-\eta(0)}{\eta(0) b}
\end{align*}
in which case $\eta^{a,b}$ is a probability measure which we will
denote by $\eta^{(b)}$. Denote its mean by 
\begin{align*}
\mathbf{m}_b = \sum_{k=2}^\infty k \eta^{(b)}(k) = \sum_{k=2}^\infty k\eta(k) b^{k-1} = (g^\eta)'(b).
\end{align*}
The mean is a strictly increasing function of $b$ and $\mathbf{m}_0=0$. In the current work we will assume that there is a $b<\rho_\eta$ such that $\mathbf{m}_b = 1$ which is then necessarily unique. The corresponding probability measure will be denoted by $\mu$ and the condition $\mathbf{m}_b = 1$ is referred to as \emph{criticality}. The condition $b<\rho_\eta$ furthermore guarantees that $\mu$ has finite exponential moments. It is straightforward to check that the simply generated tree with weights $\eta$ has the same distribution as $\mT^\mu_n$.

\subsection{Weighted trees and maps.} \label{ss:weights}

Let $\eta$ be a weight sequence such that there exists a corresponding critical offspring distribution $\mu$. Generate a random planted tree, denoted by $\mT_n^{\mu,\ast}$, as follows:
\begin{enumerate}
	\item Let $\mT_n^\mu$ have distribution $\BGW^\mu_n$.
	\item Flip a fair coin and label $\mT_n^\mu$ by $P$ or $S$ according to the outcome of the coin flip.
	\item Plant the resulting labeled tree, i.e.~set $\mT_n^{\mu,\ast} := \partial_\tree^{-1}\mT_n^\mu$.
\end{enumerate}

\begin{definition} \label{def:randommap}
The random, weighted, rooted, 2-connected SP-map with $n$ non-root edges is defined as
\begin{align*}
	\mSP_n^\mu := \varphi^\ast(\mT^{\mu,\ast}_n).
\end{align*}
The corresponding random network is denoted by 
\[
\mN_{n}^\mu = \partial_\map \mSP_n^\mu 
= \varphi(\mT_{n}^\mu).
\] 
The networks obtained by conditioning $\mN_n^\mu$ on being a series or a
parallel network (i.e.\ on the outcome of the coin flip above) are
denoted by $\mS_n^\mu$ and $\mP_n^\mu$ respectively. 
\end{definition}

In order to lighten the notation, we will often suppress the $\mu$ and
 write e.g.~$\mSP_n = \mSP^\mu_n$, $\mN_n = \mN^\mu_n$, $\mT =
\mT^\mu$, $\mT^\ast_n = \mT^{\mu,\ast}_n$ and $\mT_n = \mT^\mu_n$.

\medskip

\begin{example}\label{ex:uniform}
	The uniformly distributed map $\mSP_n$ is obtained by choosing $\eta(k) = 1$ for all $k\geq 0$, $k\neq 1$.  In this case one finds that
	\begin{align}
		\mu(0) &= 2-\sqrt{2}, \quad \mu_1 = 0, \quad \mu(k) = \left(1-\frac{1}{\sqrt{2}}\right)^{k-1}, \qquad k\geq 2. \label{eq:uniformprob}
	\end{align}
\end{example}

\begin{example}
Let $N\geq 2$ be an integer. The  map $\mSP_n$,
which is distributed uniformly among those 
where each series and each parallel composition in the network
$\mN_{n}$  consists of exactly $N$ components, is obtained by choosing
the weights $\eta(0) = \eta(N)$ and all other weights equal to 0. In
the corresponding labeled tree $\mT_{n}$, every vertex has outdegree 0
or $N$ and the root has outdegree 1 or $N$. In this case one finds
that 
$\mu(0) = 1-\frac{1}{N}$, $\mu(N) = \frac{1}{N}$.
\end{example}

\section{Proofs of the main results}

We now turn to the proofs of Theorems \ref{te:main1} and
\ref{te:main2}.  
The main method is to compare lengths of geodesics between vertices in
a series-parallel map and between corresponding vertices in its
associated tree. For concreteness, we will work with the random
parallel networks $\mP_n$. The associated tree is then a BGW-tree
$\mT_n=\mT^\mu_n$ labelled by $P$. The proof can easily be adapted to
the maps $\mS_n$, $\mN_n$ and $\mSP_n$ with exactly the same outcome.   

Let $\mT = \mT^\mu$ be an unconditioned BGW-tree and let $\xi$ be a
random variable with distribution $\mu$. 
We will use the \emph{local limit}  $\hat{\mT}$ 
of $\mT_n$ (as $n\to\infty$), which is an infinite random tree
with a unique half-infinite path starting
from the root, called its \emph{spine}.  In $\hat{\mT}$ non-spine vertices
have offspring according to  independent copies of $\xi$, whereas
spine vertices have offspring according to independent copies of
the size-biased random variable $\hat{\xi}$ with distribution
$\Pr{\hat{\xi}=k} = k \Pr{\xi = k}$, and the successor spine vertex
is chosen uniformly at random among these children. 
For further details, we refer to Janson's
extensive review \cite{janson}.  

For each $\ell \ge 1$,
let $\mT^{(\ell)}$ denote the  tree obtained by deleting
 all descendants of the $(\ell +1)$st spine vertex of $\hat{\mT}$
and identifying the $(\ell +1)$st spine vertex with the root of an
independent copy of~$\mT$.  We view all the random trees $\mT_n$,
$\mT$, $\hat{\mT}$, and $\hat{\mT}^{(\ell)}$ for $\ell \ge 1$ as
being coloured (i.e.\ some vertices are red) 
and labeled by $P$ or $S$  according to the fair coin
flip as defined in Section \ref{s:setup}. 

We may define infinite random series-parallel maps/networks $\mSP$,
$\mN$, $\mP$ and $\mS$ in a natural way by extending the functions
$\varphi$ and $\varphi^\ast$ to  infinite trees. The
maps/networks may then be viewed as the weak local limits of the
corresponding sequences of finite maps/networks. These infinite
objects are not strictly needed in the following arguments. However, it
is convenient to refer to them occasionally, so we provide some details
of their construction in the appendix. 
 
We may generate $\mT$ by
starting with a root segment  and a list of i.i.d.~non-root
segments. Likewise, $\hat{\mT}$ can be generated from a root segment,
a  list of i.i.d.~normal non-root segments, and a list of i.i.d.~spine
non-root segments. The tree $\hat{\mT}^{(\ell)}$ additionally has a mixed segment containing the tip of the spine.

\medskip

\begin{lemma}\label{le:fem}
The number of vertices in each type of segment (root, spine and
normal) has finite exponential moments. 
\end{lemma}

\medskip

\begin{proof}
  The number of vertices in each random segment may be stochastically
  bounded by two times the total population of a subcritical branching
  process (with possibly a modified root degree) whose offspring
  distribution is light-tailed. The total population of such a process
  is light-tailed itself by~\cite[Thm. 18.1]{janson}. 
	
	We detail this argument for the root segment of $\mT^\mu$. The proofs for all other types of segments are analogous. 
	We construct a BGW-tree $\mT^\nu$ with offspring probabilites $(\nu_n)_n$ from $\mT^\mu$ as follows. Let $\mT'$ be the root segment of $\mT^\mu$. The tree $\mT^\nu$ is defined as the tree with vertex set consisting of the vertices in even generations in  $\mT'$ and vertices in $\mT^\nu$ are connected by an edge if and only if one is the grandchild of the other in $\mT'$.  Note that since $\mT'$ has no vertices of outdegree 1, the number of vertices in $\mT'$ is less than two times the number of its leaves and $\mT^\nu$ has the same number of leaves as $\mT'$.  Therefore, the number of vertices in $\mT'$ is bounded by two times the number of leaves in $\mT^\nu$.

Another way to describe $\mT^\nu$ is the following.  Let $v$ be a vertex in $\mT^\mu$ at an even generation whose children we have not yet generated and assume that $v$ is also present in $\mT^\nu$. If $v$ either has no children in $\mT^\mu$ or it has more than one child and some of its children have no children, then we say that $v$ has no children in $\mT^\nu$. On the other hand, if $v$ has children in $\mT^\mu$ and all of its children also have children, then we add these grandchildren as the children of $v$ in $\mT^\nu$. From this description, we find that 
	\begin{align*}
		\nu(0) &= 1-\sum_{n=2}^\infty \mu(n) (1-\mu(0))^n, \\
		\nu(1) &= 0, \\
		\nu(k) &= \sum_{n=2}^\infty \mu(n) \sum_{i_1+\cdots+ i_n = k}\prod_{j=1}^n \mu(i_j)\mathbbm{1}\{i_j > 0\}, \quad k\geq 2.
	\end{align*}
	The generating function $g$ of the probabilities $(\nu_n)_n$ satisfies
	\begin{align}
		g(z) = \nu(0) + h(h(z))
	\end{align}
	where $h(z) = f(z) - \mu_0$ and $f$ is the generating function of the sequence $(\mu_n)_n$. 
	We then have
	\begin{align*}
		g'(1) = h'(h(1))h'(1) = h'(1-\mu(0)) = f'(1-\mu(0)) < 1
	\end{align*}
	since $f'(1) = 1$. Thus the $\nu$-process is subcritical.

	Since $\mu$ has finite exponential moments, we find that  $\nu$ also has finite exponential moments. Thus, the total population of $\mT^\nu$ (and hence also the size of the root segment of $\mT$) has finite exponential moments by~\cite[Thm. 18.1]{janson}.
\end{proof}

\begin{corollary}
	\label{co:blobbound}
	There are constants $C,c>0$ such that the probability for the maximal size  of a segment in $\mT_n$ to be larger than $x$ is bounded by $C n^{5/2} \exp(-cx)$ uniformly for all $n \ge 2$ and  $x >0$.
\end{corollary}
\begin{proof}	
The probability  for $\mT$ to have $n$ leaves satisfies
	\begin{align}
		\label{eq:partition}
	\mathbb{P}\left(|\mT|=n\right)  \sim c_{\mathrm{cond}} n^{-3/2}
	\end{align}
	for some constant $c_{\mathrm{cond}}>0$, see for example~\cite{Kortchinva}. Let $(X_i)_i$ denote the sizes of the segments in $T$.
	There are at most $n$ segments in $\mT_n$, hence the probability for the maximal size  of a segment in $\mT_n$ to be larger than $x>0$ satisfies 	\begin{align*}
		\mathbb{P}\left(\max_i X_i > x~\Big|~ |\mT|=n\right) &= \frac{\mathbb{P}\left(\bigcup_i\{X_i> x\}, |\mT|=n\right) }{\mathbb{P}\left(|\mT|=n\right) } \\
		&\leq c_{\mathrm{cond}}^{-1} n^{3/2} \sum_{i=1}^n \mathbb{P}(X_i > x),
	\end{align*}
where $|T|$ denote the number of leaves in $T$. By Lemma~\ref{le:fem}, each segment of the unconditioned tree has a
finite exponential moment and the segments are identically
distributed, except possibly the root segment. Thus each term in the
sum is bounded by a constant times $e^{-cx}$, for some $c>0$, and the
result follows.  
\end{proof}

 In the infinite map $\mP$ corresponding to the tree $\hat{\mT}$, the
 blobs $B_0, B_1, \ldots$ corresponding to the spine segments are
 called \textit{spine blobs}. The spine blobs are independent and the
 non-root spine blobs $(B_i)_{i \ge 1}$ are identically
 distributed. For each $\ell \ge 1$, consider the graph distance
 $X_\ell$ from the  south pole of $B_0$ to the south pole of $B_\ell$.  

\medskip

\begin{lemma}
  \label{le:x}
  There exists a constant $\eta>0$ such that for each $\epsilon>0$ there are constants $C,c>0$ with 
  \[
    \Pr{|X_\ell - \eta \ell| > \epsilon \ell} \le C \exp(-c\ell)
  \]
  uniformly for all $\ell \ge 1$.
\end{lemma}

\medskip

\begin{proof}
Consider the graph distance $Y_\ell$ from the south pole $v$ of $B_1$
to the south pole of $B_\ell$. This way, $|X_\ell - Y_\ell|$  is
bounded by the size of $B_0$, which has finite exponential moments by
Lemma~\ref{le:fem}. It hence suffices to show that there exists a constant 
$\eta>0$ such that for each $\epsilon>0$, there are constants $C,c>0$
such that for all $\ell \ge 1$ 
\begin{align}
	\label{eq:intermediate}
	\Pr{|Y_\ell - \eta \ell| > \epsilon \ell} \le C \exp(-c\ell).
\end{align}
Since the poles of $B_{\ell}$ are joined by an edge, 
the graph distances from $v$ to either  of them (north or south) 
differ by at most $1$.
Hence, the increment $Y_{\ell+1} - Y_{\ell}$ only depends on this
difference $S_\ell \in \{-1,0,1\}$ and $B_{\ell}$.  Thus $(Y_{\ell},
S_\ell )_{\ell \ge 1}$ is a Markov additive process with a driving
chain $(S_\ell)_{\ell\geq 1}$ whose state space has three
elements. Inequality~\eqref{eq:intermediate} now immediately follows
from a general large deviation result for Markov additive processes~\cite[Thm. 5.1]{IscNeyNum}
whose conditions are satisfied since the state space of the driving chain under consideration is finite~\cite[Rem. 3.5, Sec 7 (ii)]{IscNeyNum}. 
\end{proof}

\begin{remark}
We leave out  details concerning the exact transition probabilities of
the Markov additive process in the above proof but refer to the paper
by Curien, Haas and Kortchemski \cite{randissection} for a similar
situation.  In our case, the
increments of the additive component $(Y_\ell)_{\ell \geq 1}$ of the
chain are not as explicit as in \cite{randissection} 
(they are given in terms of distances between certain vertices in
a spine-blob) and therefore it is less useful to write them down.
\end{remark}

\bigskip

\begin{lemma}
  \label{le:z}
  Let $Z_\ell$ be the number of segments encountered on the spine from
  the root of $\hat{\mT}$ to the $(\ell +1)$st spine vertex.  
  There exists a constant $\kappa>0$ such that for each $\epsilon>0$
  there are constants $C,c>0$ such that for all $\ell \ge 1$ 
  \[
    \Pr{|Z_\ell - \kappa \ell| > \epsilon \ell} \le C \exp(-c\ell).
  \]
\end{lemma}
\begin{proof}
  Let $E_m$ denote the number of edges on the spine we need to cross
  in order to get from the root vertex to the start of the $(m+1)$st
  segment. Set $E_0=0$. Hence the differences $E_m - E_{m-1}$ are
  independent for $\ell \ge 1$ and identically distributed for all $m
  \ge 2$. They have a geometric distribution since they may be found
  by flipping independent coins to check whether the next even vertex on
  the spine is a red vertex in which case the segment stops.  
			
We make use of the following  inequality given for example
in~\cite[Example 1.4]{MR3309619}: 
Let $S_n$ denote the sum of $n$ i.i.d.~real-valued centred random
variables with finite exponential moments. Then there are constants
$\lambda_0, c>0$ such that for all $n\in \ndN$, $x > 0$ and 
$0\le\lambda\le\lambda_0$ it holds that  
\begin{equation}\label{eq:medium}
  \Pr{|S_n| \ge x} \le 2 \exp(c n \lambda^2 - \lambda x).
\end{equation}
It follows that there exists $\kappa^{-1}>0$ such that for each
$\delta>0$ there exist constants $C,c>0$ with  
\[
  \Pr{|E_m - \kappa^{-1} m| > \delta m} \le C \exp(-c m),
\qquad\text{for all } m\geq 1.
\]
As $\delta>0$ is arbitrary, it follows by choosing
$\delta = \pm \left(1-\frac{1}{1\pm\epsilon}\right)\kappa^{-1}$ that
for any $\epsilon>0$ there exist constants $C', c'>0$ with 
\[
  \Pr{\ell \notin [ E_{\lfloor \kappa \ell (1-\epsilon) \rfloor}, 
E_{\lfloor \kappa \ell (1+\epsilon) \rfloor}] }\le C' \exp(-c'\ell).
\]
Since $\ell \notin [ E_{\lfloor \kappa \ell (1-\epsilon) \rfloor},
E_{\lfloor \kappa \ell (1+\epsilon) \rfloor}$ is equivalent to
$|Z_\ell - \kappa \ell| > \kappa \epsilon \ell$, the proof is
complete. 
\end{proof}

Before proceeding with the proof, we recall some facts about the
Gromov--Hausdorff metric  $d_{\mathrm{GH}}$ on the space of isometry
classes of compact metric spaces. Informally, two compact metric
spaces are close in the Gromov--Hausdorff metric $d_{\mathrm{GH}}$ if
they may be isometrically embedded in the third space $Z$ so that
their Hausdorff distance in $Z$ is small. Here, we will use a
representation of $d_{\mathrm{GH}}$ in terms of distortions of
correspondences, see e.g.~\cite{metricgeo}. A correspondence $\cR$
between two metric spaces  $(A,d_A)$ and $(B,d_B)$ is a subset
$\cR \subseteq A \times B$ such that for any $a \in A$ there is a $b \in B$
such that $(a,b)\in\cR$, and for any $b\in B$ there is an $a \in A$
such that $(a,b)\in \cR$. A distortion of a correspondence $\cR$ is
defined by 
\begin{align*}
	\dis(\cR) = \sup\{|d_A(a_1,a_2)-d_B(b_1,b_2)|~:~(a_1,b_1),(a_2,b_2)\in \cR\}.
\end{align*}
The Gromov--Hausdorff metric may then be written as
\begin{align*}
	d_{\mathrm{GH}}(A,B) = \frac{1}{2}	\inf _\cR \dis(\cR).
\end{align*}
By the main theorem of~\cite{Kortchinva,rizzolo}, there exists a
constant $c_{\mathrm{tree}}>0$ such that  
\begin{align}\label{eq:igko}
  (\mT_n, c_{\mathrm{tree}} n^{-1/2} d_{\mT_n}) \convdis 
(\cT_{\mathrm{e}}, d_{\cT_{\mathrm{e}}})
\end{align}
in the Gromov--Hausdorff topology, with $(\cT_{\mathrm{e}},
d_{\cT_{\mathrm{e}}})$ denoting Aldous' Brownian tree. 
The rest of the proof involves defining a suitable correspondence
between the spaces $\mP_n$ and $\mT_n$, and showing that the distortion
of that correspondence, with properly rescaled metrics, converges to 0
in probability.  

We define a correspondence $\mR_n$ between $\mP_n$ and $\mT_n$ as
follows. Note that there is a bijective relationship between blobs in
$\mP_n$ and segments in $\mT_n$. We let $(u,v) \in \mR_n$ if and only
if there is a blob in $\mP_n$ which contains $u$ and a corresponding
segment in $\mT_n$ which contains $v$.  
In that case we  write $u \sim v$. 

Let $\he_{\mP_n}(u)$ denote the graph distance in the map $\mP_n$ from
the south pole $\ast_0$ to $u$, and $\he_{\mT_n}(v)$ the height of the
vertex $v$ in $\mT_n$. The following lemma is the first step in
relating distances in $\mP_n$ to distances in $\mT_n$. 

\medskip

\begin{lemma}  \label{le:heightbound}
Let $\eta$ be the constant in Lemma \ref{le:x}, $\kappa$ the constant
in Lemma \ref{le:z}, and  let $\delta>0$ be arbitrary.   
With  probability tending to one
as $n \to \infty$, there exist no vertices $u \in\mP_n$
and $v \in \mT_n$ such that $u\sim v$ and
\begin{align}
  \label{eq:main}
  |\he_{\mP_n}(u) - \eta \kappa \he_{\mT_n}(v)| \ge \delta \sqrt{n}.
\end{align}
\end{lemma}
\begin{proof}
	Let $\epsilon>0$ be given. Let $H_n = H\sqrt{n}$ for some $H>0$. By~\eqref{eq:igko} we may choose $H$ sufficiently large such that the height $\He(\mT_n)$ satisfies
	\[
	\Pr{\He(\mT_n) > H_n} \le \epsilon
	\]
	for all $n$. 
	Let us call a vertex $v$ in a tree $\tau$  \emph{bad} if it corresponds to some vertex $u$ in the associated map $P$ such that inequality~\eqref{eq:main} holds (with obvious replacements of $\mP_n$ and $\mT_n$ by $P$ and $\tau$). Let $\cE(\tau)$ denote the property that a tree $\tau$ contains a bad vertex.
Let us denote the tip of the spine in $\hat \mT^{(\ell)}$ by $U_\ell$. 	Given any finite plane tree $\tau$ with a vertex $u$ having height $\ell$, it holds that 
\begin{align*}
	\mathbb{P}\left((\hat{\mT}^{(\ell)},U_\ell)=(\tau,u)\right) = \mathbb{P}\left(\mT= \tau\right)
\end{align*}
(see e.g.~the short derivation in  \cite[Appendix A3]{seni:2023}).
 Hence, using~\eqref{eq:partition} we obtain
	\begin{align*}
		\Pr{\cE} &\le \epsilon + \Pr{\cE(T_n), \He(\mT_n) \le H_n} \\
		&\le \epsilon +  \frac{2n }{\mathbb{P}(|\mT|=n)} \sum_{\ell =0}^{\lfloor H_n \rfloor} \Pr{U_\ell \text{ is bad}}  \\
		&= \epsilon + O(n^{5/2})  \sum_{\ell =0}^{\lfloor H_n \rfloor} \Pr{U_\ell \text{ is bad}}.
	\end{align*}
	The second inequality is explained as follows. We may write
	\begin{align*}
	&\Pr{\cE(T_n), \He(\mT_n) \le H \sqrt{n}}\mathbb{P}(|\mT|=n) = \sum_{\tau \in \cT} \one\{\cE(\tau),\He(\tau)\leq H_n,|\tau|=n\} \mathbb{P}(\mT= \tau) \\
	&\leq \sum_{\tau \in \cT, u\in \tau} \one\{u \text{ is bad},\He(\tau)\leq H_n,|\tau|=n\} \mathbb{P}(\mT= \tau) \\
	&= \sum_{\tau \in \cT, u\in \tau}\sum_{\ell=0}^{\He(\tau)} \one\{u \text{ is bad},\He(\tau)\leq H_n,|\tau|=n,\he_{\tau}(u)=\ell\} 	\mathbb{P}\left((\hat{\mT}^{(\ell)},U_\ell)=(\tau,u)\right) \\
	&\leq 2n \sum_{\ell = 0}^{\lfloor H_n \rfloor}\sum_{\tau \in \cT} \mathbb{P}\left(\hat{\mT}^{(\ell)} = \tau,U_\ell \text{ is bad}\right) = 2n \sum_{\ell = 0}^{\lfloor H_n \rfloor} \mathbb{P}\left(U_\ell \text{ is bad}\right).
	\end{align*}
		For the last inequality we used the fact that if a plane tree has $n$ leaves and no inner vertex with exactly one child then it has at most $2n$ vertices in total.

	If $u$ is a vertex in the map corresponding to $\hat{\mT}^{(\ell)}$ such that $u$ corresponds to the tip of the spine of $\hat{\mT}^{(\ell)}$, then the height of $u$ may be bounded by the total number of leaves in the root segments of the subtrees attached to the spine vertices of $\hat{\mT}^{(\ell)}$. Denote by $W_i$ the total size of all the root segments in the subtrees attached to the spine vertex at distance $i$ from the root, where $0 \leq i \leq \ell$. Note that they are independent and i.i.d~for $i<\ell$. At each of these spine vertices the total number of subtrees is distributed as $\hat \xi-1$ or $\hat \xi$ (at the tip of the spine) and is thus light-tailed.  Furthermore, the total size of the root segment of each of the subtrees is light-tailed by Lemma \ref{le:fem}. We conclude that $W_i$ is light-tailed for all $0 \leq i \leq \ell$. 
	
	If $\ell <\delta / (2 \eta\kappa)  \sqrt{n}$, then the event that $U_\ell$ is bad implies that the height of $u$ is larger than $\delta \sqrt{n}$, and therefore that
	\begin{align*}
	\sum_{i=0}^\ell W_i \geq \delta \sqrt{n}.
	\end{align*}
	  Since the $W_i$ are i.i.d~and light-tailed, it follows by the medium deviation inequality~\eqref{eq:medium} that if we take $h >0$ small enough then there exists constants $C_1, c_1>0$  such that uniformly for all $n$ and $1 \le \ell \le h\sqrt{n}$
	\begin{align}
		\Pr{U_\ell \text{ is bad}} \leq \mathbb{P}\left(\sum_{i=0}^{h\sqrt{n}} W_i \geq \delta \sqrt{n}\right)\le C_1 \exp(-c_1 \sqrt{n}).
	\end{align}
	This ensures
	\[
	O(n^{5/2})  \sum_{\ell =0}^{\lfloor h \sqrt{n} \rfloor} \Pr{U_\ell \text{ is bad}} = O(n^3) \exp(-c_1 \sqrt{n}) = \exp(-\Theta(\sqrt{n})).
	\]

	Furthermore, uniformly for $h \sqrt{n} \le \ell \le H \sqrt{n}$ we know by Lemma~\ref{le:z} that for any $c>0$ we have $|Z_\ell - \kappa \ell| \le c \ell$ with probability at least $1 - \exp(\Theta(\sqrt{n}))$, which allows us to apply Lemma~\ref{le:x} to obtain for any $c>0$ that $|X_{Z_\ell} - \eta \kappa \ell | \le c \ell$ again with probability at least $1 - \exp(\Theta(\sqrt{n}))$. Recall that we constructed $\hat{\mT}^{(\ell)}$ from $\hat{\mT}$ by replacing the descendants of the tip of the spine, therefore creating a  ``mixed'' segment that contains the tip of the spine. Hence, in the map corresponding to $\hat{\mT}^{(\ell)}$ any vertex corresponding to the tip of the spine has a height that differs from $X_{Z_\ell}$ at most by a summand that by Lemma~\ref{le:fem} has finite exponential moments. Thus, we obtain 
	\[
	O(n^{5/2})  \sum_{h \sqrt{n} \le \ell \le H \sqrt{n}} \Pr{U_\ell \text{ is bad}} = O(n^3) \exp(-\Theta(\sqrt{n})) = \exp(-\Theta(\sqrt{n})).
	\]
	In summary,
	\[
	\Pr{\cE} \le \epsilon + \exp(-\Theta(\sqrt{n})).
	\]
	Since $\epsilon>0$ was arbitrary this completes the proof.
	\end{proof}

We are now ready to prove our main theorems.

\medskip

\begin{proof}[Proof of Theorem~\ref{te:main1}]
By Corollary~\ref{co:blobbound} and Lemma~\ref{le:heightbound} we know
that, for some constant $C>0$  and some deterministic sequence
$t_n =o(\sqrt{n})$, we have with high probability that the largest blob in
$\mP_n$ has diameter at most $C \log n$, and whenever a vertex $u$ in
$\mP_n$ corresponds to a vertex $v$ in $\mT_n$, then  
	\begin{align}
		\label{eq:lca}
			|\he_{\mP_n}(u) - \eta \kappa \he_{\mT_n}(v)|  \le t_n.
	\end{align}
	Now, let $u_1, u_2$ be vertices in $\mP_n$ and let $v_1$ and $v_2$ be any vertices in $\mT_n$ such that $v_i$ corresponds to $u_i$ for $i=1,2$. If $v_3$ denotes the lowest common ancestor of $v_1$ and $v_2$ in $\mT_n$, then
	\[
		d_{\mT_n}(v_1, v_2) = \he_{\mT_n}(v_1) + \he_{\mT_n}(v_2) - 2\he_{\mT_n}(v_3).
	\]
	Any geodesic in $\mP_n$ between $u_1$ and $u_2$ must pass through a blob corresponding to $v_3$. Letting $u_3$ denote some vertex in this blob, it follows that
	\[
			d_{\mP_n}(u_1, u_2) = \he_{\mP_n}(u_1) + \he_{\mP_n}(u_2) - 2\he_{\mP_n}(u_3) + O(\log n).
	\]
	By Inequality~\eqref{eq:lca} it follows that
	\[
	d_{\mP_n}(u_1, u_2) = \eta \kappa d_{\mT_n}(v_1, v_2) + o(\sqrt{n}).
	\]
	By Equation~\eqref{eq:igko} it follows that
	\begin{align}
		(\mP_n, \frac{c_{\mathrm{tree}}}{\eta \kappa} n^{-1/2} d_{\mP_n}) \convdis (\cT_{\mathrm{e}}, d_{\cT_{\mathrm{e}}})
	\end{align}
	in the Gromov--Hausdorff topology.
\end{proof}

\begin{proof}[Proof of Theorem~\ref{te:main2}]
	Since the diameter is bounded by twice the height with respect to the origin of the root edge, it suffices to show that there exist $C,c>0$ with
\[
  \Pr{\He(\mP_n) > x\sqrt{n}} \le C \exp(-c x^2 ).
\]
By possibly adjusting the constants $C,c$, it suffices to show
this inequality for all $x \ge 1$. Furthermore, the left-hand
side equals to zero if $x \ge \sqrt{n}$. Hence, we only need to
treat the case $1 \le x \le \sqrt{n}$,
which we are now going to do.

We now use the same notation and argue analogously as in the proof of
Lemma~\ref{le:heightbound}, with the exception that we call a vertex $v$ in $\mT_n$ \emph{bad} if it corresponds to some vertex $u$ in $\mP_n$ such
its height in $\mP_n$ is larger than $x \sqrt{n}$. 	 Analogously,
we  define for any $\ell \ge 1$ whether the tip of the spine of
$\hat{\mT}^{(\ell)}$ is \emph{bad}, and we let $\cE$ denote the event
that $\mT_n$ contains a bad vertex.

By~\cite[Lem. 6.61]{stufler:2020} there exist $C_1, c_1>0$ with
\[
\Pr{\He(\mT_n) > y \sqrt{n}} \le C_1 \exp(-c_1 y^2 )
\]
for any $y>0$. Hence, it follows as in the proof of
Lemma~\ref{le:heightbound} that for any fixed $h>0$ 
\begin{align*}
			\Pr{\cE} &\le \Pr{\He(\mT_n) > (x/h) \sqrt{n}}+ \Pr{\cE, \He(\mT_n) \le (x/h) \sqrt{n}} \\
		&\le C_1 \exp(-c_1 x^2/h^2) +  O(n^{5/2}) \sum_{\ell =0}^{\lfloor (x/h) \sqrt{n} \rfloor} \Pr{U_\ell \text{ is bad}}.
\end{align*}
Furthermore, applying Inequality \eqref{eq:medium} as before
yields that there exist constants $\lambda_0, \mu_{\mathrm{blob}}, c_2>0$
such that for all $n\in \ndN$ and $0 \le\lambda\le\lambda_0$ it holds
that  
\begin{align*}
  \Pr{U_\ell \text{ is bad}} \le  
2 \exp(c \ell \lambda^2 - \lambda (x \sqrt{n} - \ell \mu_{\mathrm{blob}})).
\end{align*}
As $h>0$ was fixed but arbitrary, we may take $h$ large enough such
that $\mu_{\mathrm{blob}} / h < 1/2$. As $x \ge 1$, it follows that
uniformly for $\ell \le (x/h)\sqrt{n}$ 
\[
c \ell \lambda^2 - \lambda (x \sqrt{n} - \ell \mu_{\mathrm{blob}}) \le c \ell \lambda^2 - \lambda x \sqrt{n} /2 = -x \Theta(\sqrt{n}).
\]
Hence,
\[
\Pr{\He(\mP_n) > x\sqrt{n}} = \Pr{\cE} \le C_1 \exp(-c_1 x^2/h^2) + \exp(-x \Theta(\sqrt{n})).
\]
Since we assumed $1\leq x\leq \sqrt n$, it follows that 
\[
	\Pr{\He(\mP_n) > x\sqrt{n}}   \le C_3 \exp(-c_3 x^2)
\]
for some constants $C_3, c_3 >0$ that do not depend on $n$.
\end{proof}

\section{Conclusions and related classes}

Our results imply in particular a scaling limit of uniformly sampled $2$-connected rooted series-parallel maps. From a different viewpoint, it has been shown in \cite{BonBouFus} that rooted series-parallel maps are in bijection with permutations avoiding the patterns 2413 and 3142. As a side result, it was also shown that a bipolar planar map admits a unique bipolar orientation precisely when it is series-parallel. These observations link series-parallel maps to the large framework of separable permutations and their connections to the Liouville quantum gravity, which has enjoyed remarkable developments during the last few years. For example, it is conjectured in \cite{Bor23} that conformally embedded rooted series-parallel maps converge towards the 2-LQG quantum sphere. Therefore, we expect interesting research directions following from our work.

Apart from planar maps, we may also consider planar graphs subject to constraints. The terminology for graph classes is a bit at odds with the one used here for planar maps. The class of series-parallel graphs refers to (simple) graphs that do not admit the $K_4$ as a minor, which may very well be separable (that is, not $2$-connected). This class is an important example of a subcritical graph class and hence admits the Brownian tree as scaling limit by~\cite{Pana}.

\subsection*{Acknowledgement}  
The research is partially supported by the Icelandic Research Fund,
grant number  239736-051, and the Austrian Science Fund, grant number
10.55776/F1002. 
The research of JB was supported by Vetenskapsrådet, 
grant 2023-05002, and by \emph{Ruth och Nils Erik Stenbäcks stiftelse}. The research of JT was partially supported by the Emil Aaltonen Foundation, grant number 240215.
JB and S\"OS gratefully acknowledge hospitality at the University of
Vienna and the Technical University of Vienna at the start of this
project. 
JT acknowledges the Austrian Science Fund (FWF)
under grant P33083 (DOI 10.55776/P33083) for supporting their visit. He
also thanks William da Silva for introducing the reference
\cite{BonBouFus}. 

\appendix
\section{Local metrics and continuity} \label{ss:metric}

Recall the sets $\cT$ (finite plane trees with no vertices of
outdegree 1) and $\cT^\ast$ (trees in $\cT$ where the leftmost child
of the root is a leaf) from Section \ref{sec:planetrees}.
Denote by $\overline{\cT}$
the set of (possibly infinite)
plane trees which have  no vertices of outdegree 1,
and by $\cT_\infty = \overline{\cT}\setminus \cT$  the infinite trees.
Denote by $\overline{\cT}^\ast$  the set of trees in
$\overline{\cT}$ such that the leftmost
child of the root is a leaf and let $\cT^\ast_\infty =
\overline{\cT}^\ast \setminus \cT^\ast$.
As for $\cT^\ast$, we take trees in $\overline\cT^\ast$ to be
labelled so that vertices in even generations have label $P$ and
vertices in odd generations have label $S$.

We define  \emph{local metrics} and  \emph{local topologies} on
$\overline \cT$, $\overline \cT^\ast$ and $\SP$ as
follows. First, for any graph $G$ with a distinguished vertex $v$ we
let $B_r(G;v)$ 
denote the subgraph of $G$ spanned by
vertices at distance $\leq r$ from the vertex $v$.  
For two graphs $G$ and $G'$ belonging to some set of graphs $\cG$, we let
\begin{align*}
d_{\cG}(G,G')=\frac{1}{1+\max\{r\in\mathbb{N}: B_r(G;v)=B_r(G';v)\}}.
\end{align*}
The topology generated by this metric
is called the \emph{local topology} with respect to $v$. In the case
when $\cG$ is one of the sets $\overline \cT$, 
$\overline \cT^\ast$ we take $v=\varnothing$ and use the notation $d_\cT$ for
the metric in both cases. When $\mathcal{G}$ is the set $\SP$ we use
$v=\ast_0$.  

Using the local metric we can define \emph{infinite} series-parallel
maps.  Namely, we let $\overline \SP$ denote the completion of $\SP$
with respect to the metric $d_{\SP}$. Then
$\SP_\oo:=\overline\SP\setminus\SP$ is the set of infinite
series-parallel maps, which are formally equivalence classes of Cauchy
sequences of finite series-parallel maps. 
We now show how to extend the bijection $\varphi^\ast:\cT^\ast\to\SP$
of Section \ref{ss:bijection} to infinite trees and maps.

A \textit{spine} of an infinite tree is an infinite path
$(V(0),V(1),V(2),...)$ such that $V(0)=\varnothing$ and $V(i+1)$ is a
child of $V(i)$ for all $i\geq 0$. At every vertex $V(i)$ on the spine, there
is one subtree rooted on the left of the spine and another one on the
right of the spine, which we refer to as the left and right
\emph{outgrowths} of the spine.
Let $\hat{\cT}^\ast_\infty$ be the subset of
$\cT^\ast_\infty$ consisting of trees with a \emph{unique} spine,
and such that this unique spine has 
infinitely many outgrowths in odd generations on both the left and
the right.  

Now let $T\in \hat{\cT}^\ast_\infty$ with  unique spine 
$(V_{T}(0),V_{T}(1),V_{T}(2),...)$. 
We inductively define a sequence $(i_k,j_k)_{k\geq1}$ of pairs of 
\emph{odd}
integers $i_k\leq j_k$ as follows, see Figure \ref{f:RL-pairs} for an
illustration.  First, $i_1$ is the smallest odd
integer $i$ so that $V_T(i)$ has a non-empty
outgrowth on the right of the spine, and $j_1$ is the smallest odd
integer $j\geq i_1$ so that $V_T(j)$ has a non-empty
outgrowth on the left of the spine.  Inductively,
 $i_{k+1}$ is the smallest odd
integer $i>j_k$ so that $V_T(i)$ has a non-empty
outgrowth on the right of the spine, and $j_{k+1}$ is the smallest
odd integer $j\geq i_{k+1}$ so that $V_T(j)$ has a non-empty
outgrown on the left of the spine.  
We call $(V_T(i_k),V_T(j_k))_{k\geq1}$ the
\emph{RL-pairs} of $T$.  
Since $i_k,j_k$ are odd, $V_T(i_k),V_T(j_k)$ are labelled by $S$. 
Note that $\hat{\cT}^\ast_\infty$ is defined
so that $T$ has infinitely many RL-pairs.

\begin{figure}
\begin{center}
  \includegraphics[width=0.85\linewidth]{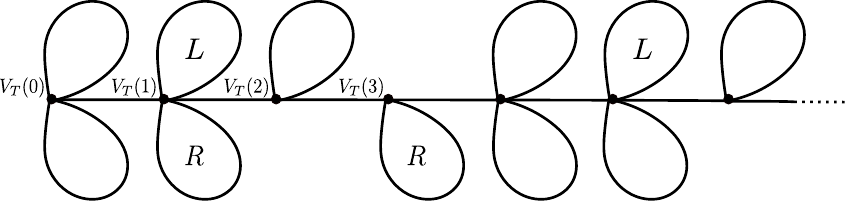}
\end{center}
\caption{An infinite tree $T\in \hat{\cT}^\ast_\infty$ 
  with its first two RL-pairs.}\label{f:RL-pairs}
\end{figure}

\medskip

\begin{proposition}
  Let $[T]_m := B_m(T;\varnothing)$ denote the first 
$m$ generations of  $T$.
\begin{enumerate}\label{prop:cont}
\item[(a)] If $T \in \hat{\cT}^\ast_\oo$ then 
$\varphi^\ast([T]_m)$ is a Cauchy sequence with respect to 
$d_\SP$.  Thus
$\varphi^\ast(T):= \lim_{m\rightarrow \infty}
\varphi^\ast(\left[T\right]_{m})$
gives a well-defined function 
$\cT^\ast\cup\hat{\cT}^\ast_\infty \to \overline\SP$.
\item[(b)] If a sequence $(T_{n})_{n\geq1}$ in 
$\cT^\ast\cup  \hat{\cT}^\ast_\oo$ 
converges to  $T\in \cT^\ast\cup  \hat{\cT}^\ast_\oo$ 
with respect to the metric $d_\cT$
then $\varphi^\ast(T_{n}) \rightarrow \varphi^\ast(T)$ with respect to 
$d_\SP$.
	\end{enumerate}
\end{proposition}

\medskip

\begin{proof}
The key observation is that if
$T \in \hat{\cT}^\ast_\infty$ and $v\in T$ is a leaf  that is
beyond the first  $r$  RL-pairs, then for any $m$ large enough such that
$v\in[T]_m$, the edge $\mathbf{e}_v$ corresponding to $v$
in $\varphi^\ast([T]_m)$
is at least at distance $r$
from both poles $\south$, $\north$.
This follows from the recursive description of $\varphi^\ast$, 
where the non-empty outgrowths of each RL-pair become non-empty
networks placed in series on either side between $\mathbf{e}_v$ and the poles
$\south$, $\north$. Each of these networks contributes to the
distance from the poles to $\mathbf{e}_v$ at least by one, and to the total distance at
least by $r$.
Due to this observation, the proof can be completed as follows.
\begin{enumerate}
\item[(a)] If $[T]_m$ contains the first $r+1$ RL-pairs for
  all $m >m_{0}(r)$ then $\forall m,m^{\prime}>m_{0}(r)$,
  $B_r(\varphi^\ast([T]_m)=B_r(\varphi^\ast([T]_{m'}))$ since all the
  edges beyond the first $r+1$ RL-pairs do not belong to
  $B_r(\varphi^\ast([T]_m))$.  
\item[(b)] Let $(T_{n})_{n\geq1}$ be a sequence in $\overline{\cT}^\ast$
  converging to $T \in \cT^\ast \cup \hat{\cT}^\ast_\oo$. If $T \in
  \cT^\ast$ then $T_{n}=T$ for every $n$ large enough so that
  $\varphi^\ast(T_{n}) \rightarrow \varphi^\ast(T)$ is obvious. Now we
  assume that $T \in \hat{\cT}^\ast_\oo$. Let $R$ be such that $[T]_R$
  contains the first $r+1$ RL-pairs. There exists a number
  $n_{0}(R)$ so that for some $n>n_{0}(R)$, $[T_{n}]_R= [T]_R$. For
  such $n$ it holds that
  $B_r(\varphi^\ast([T_{n}]_R))=B_r(\varphi^\ast([T]_R))$.\qedhere
\end{enumerate}
\end{proof}

Since the local limit $\hat{\mT}$ of the trees $\mT_n$ almost surely
belongs to $\hat\cT^*_\infty$, it follows that the local limit 
$\mSP_\oo=\varphi^\ast(\hat\mT)$ of $\mSP_n$ is a well-defined
element of $\SP_\infty$.

\end{document}